\documentclass[portuges,12pt,letter]{article}
\usepackage[centertags]{amsmath}
\usepackage{amsfonts}
\usepackage{newlfont}
\usepackage{amscd}
\usepackage{graphics}
\usepackage{epsfig}
\usepackage{indentfirst}
\usepackage{amsxtra}
\usepackage[latin1]{inputenc}
\usepackage{amssymb, amsmath}
\usepackage{amsthm}
\usepackage[mathscr]{eucal}

\newtheorem{thm}{Theorem}[section]

\newtheorem{obe}[thm]{Remark}
\title{\bf On duality principles for non-convex variational models applied to a Ginzburg-Landau type equation }
\author{ Fabio Silva Botelho}
\begin{document}

\maketitle
\abstract{ This article develops a duality principle applicable to a large class of variational problems.

Firstly, we apply the results to a Ginzburg-Landau type model. In a second step, we develop another duality principle and related primal dual variational formulation and such an approach includes optimality conditions which guarantee zero duality gap between the primal and dual formulations.

We emphasize in both cases the dual variational formulations obtained have  large regions of concavity about the critical points in question.}

\section{Introduction}
 About the year 1950 Ginzburg and Landau introduced a theory to model the super-conducting behavior of some types of materials below a critical temperature $T_c$,
which depends on the material in question. They postulated the free density energy may be written close to $T_c$ as
$$F_s(T)=F_n(T)+\frac{\hbar}{4m}\int_\Omega |\nabla \psi|^2_2 \;dx+\frac{\alpha(T)}{4}\int_\Omega |\psi|^4\;dx-\frac{\beta(T)}{2}\int_\Omega |\psi|^2\;dx,$$
where $\psi$ is a complex parameter, $F_n(T)$ and $F_s(T)$ are the normal and super-conducting free energy densities, respectively (see \cite{100, 101}
 for details).
Here $\Omega \subset \mathbb{R}^3$ denotes the super-conducting sample with a boundary denoted by $\partial \Omega=\Gamma.$ The complex function $\psi \in W^{1,2}(\Omega; \mathbb{C})$ is intended to minimize
$F_s(T)$ for a fixed temperature $T$.

Denoting $\alpha(T)$ and $\beta(T)$ simply by $\alpha$ and $\beta$,  the corresponding Euler-Lagrange equations are given by:
\begin{equation} \left\{
\begin{array}{ll}
 -\frac{\hbar}{2m}\nabla^2 \psi+\alpha|\psi|^2\psi-\beta\psi=0, & \text{ in } \Omega
 \\ \\
 \frac{\partial {\psi}}{\partial \textbf{n}}=0, &\text{ on } \partial\Omega.\end{array} \right.\end{equation}
This last system of equations is well known as the Ginzburg-Landau (G-L) one in the absence of a magnetic field and respective potential.
\begin{obe} For an open bounded subset $\Omega \subset \mathbb{R}^3$, we denote the $L^2(\Omega)$ norm by $\| \cdot \|_{L^2(\Omega)}$ or simply by $\| \cdot \|_2$.
Similar remark is valid for the $L^2(\Omega;\mathbb{R}^3)$ norm, which is denoted by $\|\cdot\|_{L^2(\Omega;\mathbb{R}^3)}$ or simply by $\| \cdot \|_2$, when its meaning is clear.%, and for  the $L^4(\Omega)$ one, which is denoted  by $\| \cdot \|_{L^4(\Omega)}$ or simply by $\| \cdot \|_4$.
 On the other hand, by $| \cdot |_2$ we denote the standard Euclidean norm in $\mathbb{R}^3$ or $\mathbb{C}^3$.
% $|\Omega|$ denotes the Lebesgue measure of $\Omega$, and  $\textbf{n}$ is the outward normal to its boundary.

Moreover derivatives are always understood in the distributional sense. Also, by a regular Lipschitzian boundary $\partial \Omega=\Gamma$ of $\Omega$, we mean regularity enough  so that the standard Sobolev Imbedding  theorems, the trace theorem and Gauss-Green formulas of integration by parts to hold. Details about such results may be found in \cite{1}.

Finally, in general $\delta F(u,v)$ will denote the Fr\'{e}chet derivative of the functional $F(u,v)$ at (u,v),  $$\delta_u F(u,v) \text{ or }\frac{\partial F(u,v)}{\partial u}$$ denotes the first Fr\'{e}chet derivative of $F$ relating the variable $u$
and $$\delta^2 F_{u,v}(u,v) \text{ or } \frac{\partial^2 F(u,v)}{\partial u \partial v}$$ denotes the second one relating the variables $u$ and $v$, always at $(u,v).$
\end{obe}

\section{The main result}

In this section we develop a duality principle for a Ginzburg-Landau type system. Such  results are a  generalization of some ideas  originally found in \cite{29} and \cite{12a}.

Similar problems are addressed in \cite{12}.

At this point we highlight such a duality principle is for  specific critical points. We start with the simpler real case.
\begin{thm}Let $\Omega \subset \mathbb{R}^3$ be an open, bounded and connected set with a regular (Lipschitzian) boundary
denoted by $\partial \Omega.$ Consider the functional $J:U \rightarrow \mathbb{R}$ given by
\begin{eqnarray}J(u)&=&\frac{\gamma}{2} \int_\Omega \nabla u \cdot \nabla u\;dx \nonumber \\ &&+\frac{\alpha}{2}\int_\Omega (u^2-\beta)^2\;dx,
\end{eqnarray}
where $dx=dx_1dx_2dx_3,$ $\gamma>0,\; \alpha>0$ and  $\beta>0$,
$$U=W^{1,2}(\Omega),$$ and $$Y^*=L^2(\Omega).$$
Let $\tilde{J}^*:Y^* \times Y^* \rightarrow \mathbb{R}$ be defined by
\begin{eqnarray}
&& \tilde{J}^*(v_0^*,z^*)
\nonumber \\ &=&\frac{1}{2}\int_\Omega \frac{(z^*)^2}{-2v_0^*+\varepsilon}\;dx \nonumber \\ &&
-\frac{\gamma}{2}\int_\Omega \left|\nabla\left(\frac{z^*}{-2v_0^*+\varepsilon}\right)\right|^2_2\;dx
\nonumber \\ &&-\frac{1}{2}\int_\Omega \frac{\left|\gamma\nabla^{2}\left(\frac{z^*}{-2 v_0^*+\varepsilon}\right)+z^*\right|^2}{\varepsilon}\;dx \nonumber \\ &&-\frac{1}{2 \alpha} \int_\Omega (v_0^*)^2\;dx
-\beta \int_\Omega v_0^*\;dx,
\end{eqnarray}

Let $u_0 \in U$ be such that $$\delta J(u_0)=0.$$

Defining,
$$\tilde{v}_0^*=\alpha(u_0^2-\beta),$$
$$\tilde{z}^*=(-2\tilde{v}_0^*+\varepsilon)u_0,$$

Suppose, $\tilde{v}_0^* \in B,$
where $$B=\{v^*_0 \in Y^*\;:\; -2v^*_0+\varepsilon > \varepsilon^{1/8}, \text{ in } \Omega\},$$
and $\varepsilon>0$ is a small parameter.

Under such hypotheses, $(\tilde{v}_0^*,\tilde{z}^*) \in Y^* \times Y^*$ is such that
$$\delta \tilde{J}^*(\tilde{v}_0^*, \tilde{z}^*)=0.$$

Moreover, assuming from now an appropriate discretized problem version, denoting
$$L(v_0^*)z^*=\gamma\nabla^{2}\left(\frac{z^*}{-2 v_0^*+\varepsilon}\right)$$ suppose also
$$(L(\tilde{v}_0^*)+I_d)^2 > \varepsilon^{1/4},\text{ in }\Omega$$ where $I_d$ denotes the identity operator, and
$$\left(\frac{\partial L(\tilde{v}^*_0)}{\partial v_0^*} \tilde{z}^*\right)^2> \varepsilon^{1/4},  \text{ in }\Omega.$$

Under such hypotheses, there exists $r>0$ such that

\begin{eqnarray}
J(u_0)&=& \min_{u \in U} \left\{J(u)+\frac{1}{2}\int_\Omega K(\tilde{v}_0^*)(u-u_0)^2\;dx\right\} \nonumber \\ &=&
\max_{(v^*_0,z^*)\in B_r(\tilde{v}_0^*,\tilde{z}^*)}\{\tilde{J}^*(v_0^*,z^*)\} \nonumber \\ &=&
\tilde{J}^*(\tilde{v}_0^*, \tilde{z}^*),
\end{eqnarray}
where, $$K(\tilde{v}_0^*)=-2\tilde{v}_0^*+\varepsilon,$$
and $B_r(\tilde{v}_0^*,\tilde{z}^*)$ denotes the open ball of center $(\tilde{v}_0^*, \tilde{z}^*)$ and radius $r>0$.

\end{thm}
\begin{proof}
Denoting $v^*=(v_0^*,v_1^*),$ let  $J_K: [Y^*]^3 \rightarrow \mathbb{R}$ be the functional such that
\begin{eqnarray}
J_K(v^*,z^*)&=& F^*(z^*)-G_0^*(z^*,v_1^*)-(G_1)_K^*(v_1^*,v_0^*)
\nonumber \\ &=&\frac{1}{2} \int_\Omega \frac{(z^*)^2}{K}\;dx\nonumber \\ &&-\frac{1}{2\gamma} \int_\Omega (z^*-v_1^*)[(-\nabla^2)^{-1}(z^*-v_1^*)]\;dx
\nonumber \\ &&-\frac{1}{2} \int_\Omega \frac{(v_1^*)^2}{2v_0^*+K}\;dx-\frac{1}{2 \alpha}\int_\Omega (v_0^*)^2\;dx-\beta \int_\Omega v_0^*\;dx.
\end{eqnarray}
Here we denote,

$$J(u)=G_0(u)+(G_1)_K(u,0)-F(u),$$ where
$$G_0(u)=\frac{\gamma}{2}\int_\Omega \nabla u \cdot \nabla u\;dx,$$

$$(G_1)_K(u,v)=\frac{\alpha}{2}\int_\Omega (u^2-\beta+v)^2+\frac{1}{2}\int_\Omega K u^2\;dx,$$
and
$$F(u)=\frac{1}{2}\int_\Omega K u^2\;dx,$$
where $K \in L^2(\Omega)$ is a function to be specified in the next lines.

Moreover, \begin{eqnarray}
G^*_0(z^*,v_1^*)&=&\sup_{u \in U}\{\langle z^*-v_1^*, u\rangle_{L^2}-G_0(u)\} \nonumber \\ &=&
\frac{1}{2\gamma}\int_\Omega (z^*-v_1^*)[(\nabla^2)^{-1}(z^*-v_1^*)]\;dx,\end{eqnarray}
\begin{eqnarray}
G^*_1(v_1^*,v_0^*)&=&\sup_{u \in U}\{\langle v_1^*, u\rangle_{L^2}+\langle v_0^*,v\rangle_{L^2}-(G_1)_K(u,v)\} \nonumber \\ &=&
\frac{1}{2}\int_\Omega \frac{(v_1^*)^2}{2v_0^*+K} \;dx+\frac{1}{2\alpha} \int_\Omega (v_0^*)^2\;dx+\beta \int_\Omega v_0^*\;dx,\end{eqnarray}
if $$2\;v_0^*+K>0, \text{ in } \Omega.$$
Also,
\begin{eqnarray}F^*(z^*)&=&\sup_{u \in U}\{ \langle z^*,u\rangle_{L^2}-F(u)\} \nonumber \\ &=& \frac{1}{2}\int_\Omega \frac{(z^*)^2}{K}\;dx.\end{eqnarray}

In particular for $K=-2v_0^*+\varepsilon,$ we obtain
\begin{eqnarray}
J_K^*(v^*,z^*)&=& \hat{J}^*(v_0^*,v_1^*,z^*) \nonumber \\ &=& \frac{1}{2}\int_\Omega \frac{(z^*)^2}{-2v_0^*+\varepsilon}\;dx
\nonumber \\ &&-\frac{1}{2\gamma}\int_\Omega(z^*-v_1^*)[(-\nabla^2)^{-1}(z^*-v_1^*)]\;dx \nonumber \\ &&-\frac{1}{2}\int_\Omega \frac{(v_1^*)^2}{\varepsilon}-\frac{1}{2 \alpha} \int_\Omega (v_0^*)^2\;dx
-\beta \int_\Omega v_0^*\;dx.
\end{eqnarray}

Let $\hat{v}_1^*$ be the solution of equation
$$\delta_{z^*}\hat{J}^*(v_0^*,\hat{v}_1^*,z^*)=0,$$ so that
$$\hat{v}_1^*=\gamma(\nabla^2)\left( \frac{z^*}{-2v_0^*+\varepsilon} \right)+z^*.$$

So, from such a result we obtain,

\begin{eqnarray}
J_K^*(v^*,z^*)&=& \hat{J}^*(v_0^*,\hat{v}_1^*,z^*) \nonumber \\ &=& \tilde{J}^*(v_0^*,z^*)
\nonumber \\ &=&\frac{1}{2}\int_\Omega \frac{(z^*)^2}{-2v_0^*+\varepsilon}\;dx \nonumber \\ &&
-\frac{\gamma}{2}\int_\Omega \left|\nabla\left(\frac{z^*}{-2v_0^*+\varepsilon}\right)\right|^2_2\;dx
\nonumber \\ &&-\frac{1}{2}\int_\Omega \frac{\left|\gamma\nabla^{2}\left(\frac{z^*}{-2 v_0^*+\varepsilon}\right)+z^*\right|^2}{\varepsilon}\;dx \nonumber \\ &&-\frac{1}{2 \alpha} \int_\Omega (v_0^*)^2\;dx
-\beta \int_\Omega v_0^*\;dx,
\end{eqnarray}
where, as indicated above, we have denoted
$$K=K(v_0^*)=-2v_0^*+\varepsilon.$$
For $$\hat{v}_1^*=\gamma \nabla^2\left(\frac{\tilde{z}_0^*}{-2\tilde{v}_0^*+\varepsilon}\right)+\tilde{z}_0^*,$$ from $$\delta J(u_0)=0$$
and from the standard Legendre transform proprieties, we have
$$\delta J_K^*(\hat{v}_1^*,\tilde{v}_0^*, \tilde{z}^*)=0.$$

Hence,
\begin{eqnarray}
\frac{\partial \tilde{J}^*(\tilde{v}_0^*,\tilde{z}^*)}{\partial z^*} &=& \frac{\partial J_K^*(\hat{v}_1^*,\tilde{v}_0^*,\tilde{z}^*)}{\partial z^*}
\nonumber \\ &&+\frac{\partial J_K^*(\hat{v}_1^*,\tilde{v}_0^*,\tilde{z}^*)}{\partial v_1^*}\frac{\partial \hat{v}_1^*}{\partial z^*} \nonumber \\
&=& \frac{\partial J_K^*(\hat{v}_1^*,\tilde{v}_0^*,\tilde{z}^*)}{\partial z^*}
\nonumber \\ &=& 0.
\end{eqnarray}

Also,
\begin{eqnarray}
\frac{\partial \tilde{J}^*(\tilde{v}_0^*,\tilde{z}^*)}{\partial v_0^*} &=& \frac{\partial J_K^*(\hat{v}_1^*,\tilde{v}_0^*,\tilde{z}^*)}{\partial v_0^*}
\nonumber \\ &&+\frac{\partial J_K^*(\hat{v}_1^*,\tilde{v}_0^*,\tilde{z}^*)}{\partial v_1^*}\frac{\partial \hat{v}_1^*}{\partial v_0^*} \nonumber \\&&+\frac{\partial J_K^*(\hat{v}_1^*,\tilde{v}_0^*,\tilde{z}^*)}{\partial K}\frac{\partial K(\tilde{v}_0^*)}{\partial v_0^*} \nonumber \\
&=& \frac{\partial J_K^*(\hat{v}_1^*,\tilde{v}_0^*,\tilde{z}^*)}{\partial v_0^*}
\nonumber \\ &=&0,
\end{eqnarray}
since \begin{eqnarray}
\frac{\partial J_K^*(\hat{v}_1^*,\tilde{v}_0^*,\tilde{z}^*)}{\partial K}&=& -\frac{1}{2}\frac{(\tilde{z}^*)^2}{K^2}+\frac{1}{2}\frac{(\hat{v}_1^*)^2}{
(-2\tilde{v}_0^*+K)^2} \nonumber \\ &=& -\frac{u_0^2}{2}+\frac{u_0^2}{2}, \text{ in } \Omega.
\end{eqnarray}

From such results and also from the Legendre transform properties, we obtain
\begin{eqnarray}
J(u_0)&=& \tilde{J}^*(\tilde{v}_0^*,\tilde{z}^*) \nonumber \\ &=&
J_K^*(\tilde{v}_1^*,\tilde{v}_0^*, \tilde{z}^*) \nonumber \\ &=&
\frac{1}{2}\int_\Omega \frac{(\tilde{z}^*)^2}{K}\;dx
-\frac{1}{2\gamma}\int_\Omega (\tilde{z}^*-\hat{v}_1^*)[(-\nabla^2)^{-1}(z^*-\hat{v}_1^*)]\;dx
\nonumber \\ &&-\frac{1}{2} \int_\Omega \frac{(\hat{v}_1^*)^2}{2\tilde{v}_0^*+K}\;dx-\frac{1}{2 \alpha} (\tilde{v}_0^*)^2\;dx-\beta \int_\Omega \tilde{v}_0^*\;dx.
\end{eqnarray}
Thus,
\begin{eqnarray}
J(u_0)&=&\tilde{J}^*(\tilde{v}_0^*, \tilde{z}^*) \nonumber \\ &\leq&
\frac{1}{2}\int_\Omega \frac{(\tilde{z}^*)^2}{K}\;dx
\nonumber \\ && -\langle u,\tilde{z}^* \rangle_{L^2}+\frac{\gamma}{2} \int_\Omega \nabla u \cdot \nabla u\;dx
\nonumber \\ && +\frac{1}{2}\int_\Omega (2\tilde{v}_0^*+K)u^2\;dx -\frac{1}{2 \alpha} \int_\Omega (\tilde{v}_0^*)^2\;dx
\nonumber \\ && -\beta \int_\Omega \tilde{v}_0^*\;dx \nonumber \\ &\leq&
\frac{1}{2}\int_\Omega \frac{(\tilde{z}^*)^2}{K}\;dx
\nonumber \\ && -\langle u,\tilde{z}^* \rangle_{L^2}+\frac{\gamma}{2} \int_\Omega \nabla u \cdot \nabla u\;dx
\nonumber \\ && + \sup_{ v_0^* \in Y^*}\left\{\frac{1}{2}\int_\Omega (2v_0^*+K)u^2\;dx -\frac{1}{2 \alpha} \int_\Omega (v_0^*)^2\;dx
 -\beta \int_\Omega v_0^*\;dx\right\} \nonumber \\ &=&
\frac{1}{2}\int_\Omega \frac{(\tilde{z}^*)^2}{K}\;dx
\nonumber \\ && -\langle u,\tilde{z}^* \rangle_{L^2}+\frac{\gamma}{2} \int_\Omega \nabla u \cdot \nabla u\;dx
\nonumber \\ && + \frac{\alpha}{2}\int_\Omega (u^2-\beta)^2\;dx+\frac{1}{2}\int_\Omega K u^2\;dx.
\end{eqnarray}

From this,
 \begin{eqnarray}\label{pr12}
J(u_0)&=& \tilde{J}^*(\tilde{v}_0^*,\tilde{z}^*) \nonumber \\ &\leq&
\frac{1}{2}\int_\Omega Ku_0^2\;dx
\nonumber \\ && -\langle u,Ku_0 \rangle_{L^2}+\frac{\gamma}{2} \int_\Omega \nabla u \cdot \nabla u\;dx
\nonumber \\ && + \frac{\alpha}{2}\int_\Omega (u^2-\beta)^2\;dx+\frac{1}{2}\int_\Omega K u^2\;dx \nonumber \\ &=&
\frac{1}{2} \int_\Omega K(u-u_0)^2\;dx+J(u),\; \forall u \in U.
\end{eqnarray}
Moreover defining
\begin{eqnarray}J_1^*(v_0^*,z^*)&=&\frac{1}{2}\int_\Omega \frac{\left|\gamma \nabla^{2}\left(\frac{z^*}{-2 v_0^*+\varepsilon}\right)+z^*\right|^2}{\varepsilon}\;dx
\nonumber \\ &=& \frac{1}{2\varepsilon}\int_\Omega (L(v_0^*)z^*+z^*)^2\;dx
\nonumber \\ &=& \frac{1}{2\varepsilon} \int_\Omega [(L(v_0^*)+I_d)z^*]^2\;dx
\end{eqnarray}
where, as above indicated, we have denoted
$$\gamma \nabla^2\left(\frac{z^*}{-2 v_0^*+\varepsilon} \right)+z^*=L(v_0^*)z^*+z^*,$$
we have
$$\frac{\partial^2 J_1^*(\tilde{v}_0^*,\tilde{z}^*)}{\partial (z^*)^2} = \frac{(L(\tilde{v}_0^*)+I_d)^2}{\varepsilon} > \mathcal{O}\left(\frac{1}{\varepsilon^{3/4}} \right).$$

On the other hand,

\begin{eqnarray}
\frac{\partial^2J_1^*(\tilde{v}_0^*,\tilde{z}^*)}{\partial (z^*)\partial v_0^*} &=&\left( \frac{(L(\tilde{v}_0^*)+I_d)\tilde{z}^*}{\varepsilon}\right)\frac{\partial L(\tilde{v}_0^*)}{\partial v_0^*} \nonumber \\ &=&  u_0 \frac{\partial L(\tilde{v}_0^*)}{\partial v_0^*},
\end{eqnarray}
where $$u_0=\frac{(L(\tilde{v}_0^*)+I_d)\tilde{z}^*}{\varepsilon}.$$

Also,
\begin{eqnarray}
\frac{\partial^2 J^*_1(\tilde{v}_0^*,\tilde{z}^*)}{\partial (v_0^*)^2}&=& \frac{  \left[\frac{\partial L(\tilde{v}_0^*)}{\partial v^*_0} \tilde{z}^*\right]^2 }{\varepsilon} +\left(\frac{ [L(\tilde{v}_0^*)+I_d]\tilde{z}^*}{\varepsilon}\right)
\frac{\partial^2 L(\tilde{v}_0^*)}{\partial (v_0^*)^2}\tilde{z}^* \nonumber \\ &=&\frac{  \left[\frac{\partial L(\tilde{v}_0^*)}{\partial v^*_0}\tilde{z}^*\right]^2}{\varepsilon} +u_0\frac{\partial^2 L(\tilde{v}_0^*)}{\partial (v_0^*)^2}\tilde{z}^* > \mathcal{O}\left(\frac{1}{\varepsilon^{3/4}}\right).
\end{eqnarray}
Hence,
$$\det\{\delta^2_{v_0^*,z^*} J^*_1(\tilde{v}_0^*,\tilde{z}^*)\} > \mathcal{O}\left(\frac{1}{\varepsilon^{3/2}}\right),$$
so that from this we may infer that $(\tilde{v}_0^*,\tilde{z}^*)$ is a point of local maximum for $$\tilde{J}(\tilde{v}_0^*,\tilde{z}^*).$$

Therefore, there exists $r>0$ such that
$$J^*(\tilde{v}_0^*, \tilde{z}^*)=\max_{(v_0^*,z^*) \in B_r(\tilde{v}_0^*,\tilde{z}^*)} J^*(v_0^*,z^*).$$

From this and (\ref{pr12}), we obtain
\begin{eqnarray}
J(u_0)&=& \min_{u \in U} \left\{J(u)+\frac{1}{2}\int_\Omega K(\tilde{v}_0^*)(u-u_0)^2\;dx\right\} \nonumber \\ &=&
\max_{(v^*_0,z^*)\in B_r(\tilde{v}_0^*,\tilde{z}^*)}\{\tilde{J}(v_0^*,z^*)\} \nonumber \\ &=&
\tilde{J}(\tilde{v}_0^*, \tilde{z}^*),
\end{eqnarray}
where, $$K(\tilde{v}_0^*)=-2\tilde{v}_0^*+\varepsilon.$$

The proof is complete.
\end{proof}
\section{ A primal dual variational formulation and related duality principle}
In this section we develop a new duality principle. The result is summarized by the following theorem.
\begin{thm}Let $\Omega \subset \mathbb{R}^3$ be an open, bounded and connect set with a regular (Lipschitzian) boundary
denoted by $\partial \Omega.$ Consider the functional $J:U \rightarrow \mathbb{R}$ defined by
\begin{eqnarray} J(u)&=& \frac{\gamma}{2} \int_\Omega \nabla u \cdot \nabla u\;dx \nonumber \\ &&+
\frac{\alpha}{2}\int_\Omega(u^2-\beta)^2\;dx - \langle u,f \rangle_{L^2},\end{eqnarray}
where $$U=W_0^{1,2}(\Omega)$$ and $f \in L^2(\Omega).$

Assuming from now and on a discretized finite dimensional problem version, suppose $u_0 \in U$ is such that
$$\delta J(u_0)= \mathbf{0}.$$

Let $\varepsilon>0$ be a small real  parameter. Define $$\tilde{v}_0^*=\alpha(u_0^2-\beta) \in Y^*=L^2(\Omega)$$ and
assume $$(-\nabla^2+(2\tilde{v}_0^*-\varepsilon) I_d)^2 > \sqrt{\varepsilon}, \text{ in } \Omega.$$

Suppose also $f\; u_0\geq 0, \text{ in } \Omega.$

Under such hypotheses, there exists $r>0$ such that

\begin{eqnarray}J(u_0)&=&\min_{u \in U}\left\{J(u)+\frac{1}{2}\int_\Omega K(\tilde{v}_0^*)(u-u_0)^2\;dx \right\}
\nonumber \\ &=& \max_{(v_0^*, \hat{u}) \in B_r(\tilde{v}_0^*,u_0)} J_3^*(v_0^*,\hat{u}) \nonumber \\ &=&
J_3^*(\tilde{v}_0^*, u_0),
\end{eqnarray}
where
\begin{eqnarray}
J_3^*(v_0^*,\hat{u})&=& -\frac{\gamma}{2} \int_\Omega \nabla \hat{u} \cdot \nabla \hat{u} \; dx-\frac{1}{2}\int_\Omega (2v_0^*- \varepsilon)\hat{u}^2\;dx
\nonumber \\ &&-\frac{1}{2 \varepsilon} \int_\Omega (\gamma \nabla^2 \hat{u}+(-2v_0^*+\varepsilon) \hat{u}+f)^2\;dx
\nonumber \\ &&-\frac{1}{2 \alpha} \int_\Omega (v_0^*)^2\;dx-\beta\int_\Omega v_0^*\;dx
\end{eqnarray}
\end{thm}
\begin{proof}
Denote $$G_0(u)= \frac{\gamma}{2} \int_\Omega \nabla u \cdot \nabla u \;dx,$$
$$G_{1K} (u,v)= \frac{\alpha}{2} \int_\Omega (u^2 - \beta+v)^2\;dx+\frac{1}{2} \int_\Omega K u^2\;dx-\langle u,f \rangle_{L^2},$$
$$F(u)=\frac{1}{2} \int_\Omega K u^2\;dx,$$
where $K \in L^2(\Omega)$ will be specified in the next lines.

Observe that
$$J(u)=G_0(u)+G_{1K}(u)-F(u),\; \forall u \in U.$$

Let $J^*_K:[Y^*]^3 \rightarrow \mathbb{R}$ be defined as
$$J_K^*(v^*,z^*)=F^*(z^*)-G_{1K}^*(v_0^*,v_1^*)-G_0^*(v_1^*,z^*),$$
$\forall v^*=(v_0^*,v_1^*) \in [Y^*]^2, \; z^* \in Y^*,$ where
\begin{eqnarray}
F^*(z^*)&=& \sup_{u \in U}\{ \langle u, z^* \rangle_{L^2}-F(u)\} \nonumber \\ &=&
\frac{1}{2} \int_\Omega \frac{(z^*)^2}{K}\;dx. \end{eqnarray}
\begin{eqnarray} G_{1K}^*(v_0^*,v_1^*)&=& \sup_{ u \in U} \{\langle u,v_1^*\rangle_{L^2}+\langle v,v_0^* \rangle_{L^2} \nonumber \\ &&
-G_{1K}(u,v)\} \nonumber \\ &=& -\frac{1}{2} \int_\Omega \frac{(v_1^*+f)^2}{2v_0^*+K} \;dx-\frac{1}{2\alpha} \int_\Omega (v_0^*)^2\;dx \nonumber \\
&& -\beta \int_\Omega v_0^*\;dx,
\end{eqnarray}
if $2v_0^*+K>0, \; \text{ in } \Omega.$

Also
\begin{eqnarray}
G_0^*(v_1^*,z^*)&=& \sup_{u \in U}\{-\langle u,v_1^* \rangle_{L^2}+\langle u,z^* \rangle_{L^2} -G_0(u)\}
\nonumber \\ &=& \frac{1}{2 \gamma} \int_\Omega (z^*-v_1^*)[(-\nabla^2)^{-1}(z^*-v_1^*)]\;dx.
\end{eqnarray}

Let $\hat{v}_1^*$ be such that $$\delta_{z^*} J_K^*(v_0^*, \hat{v}_1^*,z^*)=\mathbf{0},$$
that is,
$$\frac{z^*}{K}- \frac{[(-\nabla)^2]^{-1}(-\hat{v}_1^*+z^*)}{\gamma}=0, \text{ in } \Omega.$$
Thus,
$$\hat{v}_1^*=\gamma(\nabla^2)\left( \frac{z^*}{K}\right) +z^*,$$
so that
\begin{eqnarray}
J_K^*(v_0^*, \hat{v}_1^*,z^*) &=& J_1^*(v_0^*,z^*,K) \nonumber \\ &=&
\frac{1}{2} \int_\Omega \frac{(z^*)^2}{K}\;dx- \frac{\gamma}{2}\int_\Omega \left|\nabla \left( \frac{z^*}{K} \right) \right|^2\;dx \nonumber \\
&&- \frac{1}{2} \int_\Omega \frac{\left( \gamma \nabla^2\left(\frac{z^*}{K}\right)+z^*+f\right)^2}{2v_0^*+K}\;dx-\frac{1}{2 \alpha} \int_\Omega (v_0^*)^2\;dx
\nonumber \\ && -\beta \int_\Omega v_0^*\;dx.
\end{eqnarray}
Specifically, for $$K=K(v_0^*)=-2v_0^*+\varepsilon,$$ we obtain
\begin{eqnarray}
J_K^*(v_0^*, \hat{v}_1^*,z^*) &=& J_1^*(v_0^*,z^*,K) \nonumber \\ &=& J_2^*(v_0^*,z^*) \nonumber \\ &=&
\frac{1}{2} \int_\Omega \frac{(z^*)^2}{-2v_0^*+\varepsilon}\;dx- \frac{\gamma}{2}\int_\Omega \left|\nabla \left( \frac{z^*}{-2v_0^*+\varepsilon} \right) \right|^2\;dx \nonumber \\
&&- \frac{1}{2\varepsilon} \int_\Omega \left( \gamma \nabla^2\left(\frac{z^*}{-2v_0^*+ \varepsilon}\right)+z^*+f\right)^2\;dx-\frac{1}{2 \alpha} \int_\Omega (v_0^*)^2\;dx
\nonumber \\ && -\beta \int_\Omega v_0^*\;dx.
\end{eqnarray}
Finally, defining $$\hat{u}= \frac{z^*}{-2v_0^*+\varepsilon},$$
that is,
$$z^*=(-2v_0^*+\varepsilon) \hat{u},$$ we obtain
\begin{eqnarray}
J_2^*(v_0^*,z^*)&=& J_3^*(v_0^*, \hat{u}) \nonumber \\ &=&
-\frac{\gamma}{2} \int_\Omega \nabla \hat{u} \cdot \nabla \hat{u} \;dx- \frac{1}{2}\int_\Omega (2v_0^*- \varepsilon) \hat{u}^2\;dx \nonumber \\
&&- \frac{1}{2\varepsilon} \int_\Omega (\gamma \nabla^2 \hat{u} +(-2v_0^*+\varepsilon) \hat{u}+f)^2\;dx-\frac{1}{2 \alpha} \int_\Omega (v_0^*)^2\;dx
\nonumber \\ && -\beta \int_\Omega v_0^*\;dx.
\end{eqnarray}
Now, from $\delta J(u_0)= \mathbf{0}$ we have,
$$\gamma \nabla^2 u_0+(-2 \tilde{v}_0^*+\varepsilon)u_0+f-\varepsilon u_0=0,$$
$\text{ in } \Omega,$ so that
\begin{eqnarray} u_0 &=& \frac{ \gamma \nabla^2 u_0+(-2\tilde{v}_0^*+\varepsilon) u_0+f}{\varepsilon} \nonumber \\ &\equiv& u_1.
\end{eqnarray}

From this and the variation of $J_3^*$ in $\hat{u}$ at $(\tilde{v}_0^*,u_0)$, we obtain
\begin{eqnarray}
\frac{\partial J_3^*(\tilde{v}_0^*, u_0)}{\partial \hat{u}}&=& \gamma \nabla^2 u_0+(-2 \tilde{v}_0^*+\varepsilon)u_0 \nonumber \\ &&
-\gamma \nabla^2 u_1-(-2 \tilde{v}_0^*+\varepsilon)u_1 \nonumber \\ &=& 0, \text{ in }\Omega.\end{eqnarray}

Moreover, from the variation of $J_3^*$ in $v_0^*$ at $(\tilde{v}_0^*,u_0)$ we have,
\begin{eqnarray}
\frac{\partial J_3^*(\tilde{v}_0^*, u_0)}{\partial v_0^*}&=& -\frac{\tilde{v}_0^*}{\alpha}-u_0^2+ 2 u_1u_0- \beta \nonumber \\ &=&
-\frac{\tilde{v}_0^*}{\alpha}+u_0^2- \beta  \nonumber \\ &=& 0, \text{ in } \Omega, \end{eqnarray}
since $$\tilde{v}_0^*=\alpha (u_0^2-\beta), \text{ in } \Omega.$$

From these last two results we may infer that
$$\delta J_3^*(\tilde{v}_0^*, u_0)=\mathbf{0}.$$

Finally,
\begin{eqnarray}
-\delta_{\hat{u} \hat{u}}^2J^*(v_0^*,u_0)&=& \gamma \nabla^2-(2v_0^*-\varepsilon )I_d \nonumber \\
&& +\frac{(\gamma \nabla^2-(2v_0^*-\varepsilon )I_d)^2}{\varepsilon} > \mathcal{O}\left(\frac{1}{\sqrt{\varepsilon}}\right).
\end{eqnarray}

Also, $$-\delta^2_{v_0^* v_0^*} J_3^*(\tilde{v}_0^*,u_0)= \frac{1}{\alpha}+ \frac{4 u_0^2}{\varepsilon}.$$

Finally,
$$\delta_{v_0^*, \hat{u}}^2J_3^*(\tilde{v}_0^*,u_0) =4u_0 -\frac{2 f}{\varepsilon}.$$

Since $$f u_0 \geq 0 \text{ in } \Omega,$$ we may compute,
$$det[ \delta^2 J_3^*(\tilde{v}_0^*,u_0)] > \mathcal{O}\left( \frac{1}{\sqrt{\varepsilon}}\right),$$ so that
$(\tilde{v}_0^*,u_0)$ is a point of local maximum for $J_3^*(v_0^*,\hat{u}).$

Therefore, there exists $r>0$ such that
\begin{equation}\label{J1}J_3^*(\tilde{v}_0^*,u_0)=\max_{(v_0^*\hat{u}) \in B_r(\tilde{v}_0^*,u_0)} J_3^*(v_0^*, \hat{u}).\end{equation}

Moreover, from the Legendre transform transform proprieties we obtain
\begin{eqnarray}
J(u_0)&=&J_3^*(\tilde{v}_0^*,u_0) \nonumber \\ &=& J_K^*(\tilde{v}_0^*, \hat{v}_1^*(u_0),z^*(u_0)),\end{eqnarray}

where $$z^*(u_0)=(-2\tilde{v}_0^*+\varepsilon)u_0,$$
and
$$\hat{v}_1^*(u_0)=\gamma \nabla^2 \left( \frac{z^*(u_0)}{-2\tilde{v}_0^*+\varepsilon} \right)+z^*(u_0).$$

Therefore,
\begin{eqnarray}
J(u_0)&=& J_K^*(\tilde{v}_0^*,\hat{v}_1(u_0),z^*(u_0)) \nonumber \\ &=&
\frac{1}{2}\int_\Omega \frac{(z_0^*(u_0))^2}{K}\;dx-G_{1K}^*(\tilde{v}_0^*, \hat{v}_1^*(u_0)) \nonumber \\ &&-
G_0^*(\hat{v}_1^*(u_0),z^*(u_0)), \end{eqnarray}
so that
\begin{eqnarray}
J(u_0) &\leq& \frac{1}{2} \int_\Omega K u_0^2\;dx- \langle u, z_0^*(u_0)\rangle_{L^2}
\nonumber \\ &&+\frac{1}{2}\int_\Omega (2\tilde{v}_0^*+K) u^2\;dx - \frac{1}{2\alpha} \int_\Omega (\tilde{v}_0^*)^2\;dx
\nonumber \\ &&-\beta \int_\Omega \tilde{v}_0^*\;dx+\frac{\gamma}{2}\int_\Omega \nabla u \cdot \nabla u\;dx  \nonumber \\ &&
-\langle u,f \rangle_{L^2}, \forall u \in U,
\end{eqnarray}
and thus,
\begin{eqnarray}
J(u_0) &\leq& \frac{1}{2} \int_\Omega K u_0^2\;dx- \langle u, K u_0\rangle_{L^2}
\nonumber \\ &&+\sup_{ v_0^* \in Y^*}\left\{\frac{1}{2}\int_\Omega (2v_0^*+K) u^2\;dx - \frac{1}{2\alpha} \int_\Omega (v_0^*)^2\;dx\right.
\nonumber \\ &&\left.-\beta \int_\Omega v_0^*\;dx\right\}+\frac{\gamma}{2}\int_\Omega \nabla u \cdot \nabla u\;dx  \nonumber \\ &&
-\langle u,f \rangle_{L^2} \nonumber \\ &=& \frac{1}{2} \int_\Omega K (u-u_0)^2\;dx +\frac{\gamma}{2}\int_\Omega \nabla u \cdot \nabla u\;dx \nonumber \\ && +\frac{\alpha}{2} \int_\Omega (u^2-\beta)^2\;dx-\langle u,f \rangle_{L^2}, \forall u \in U.
\end{eqnarray}
Hence,
$$J(u_0) \leq J(u)+\frac{1}{2} \int_\omega K (u-u_0)^2\;dx, \forall u \in U,$$
where $$K=K(\tilde{v}_0^*)=-2\tilde{v}_0^* +\varepsilon.$$

From this,  (\ref{J1}) and from $$J(u_0)=J_3^*(\tilde{v}_0^*,u_0)$$ we have,

\begin{eqnarray}J(u_0)&=&\min_{u \in U}\left\{J(u)+\frac{1}{2}\int_\Omega K(\tilde{v}_0^*)(u-u_0)^2\;dx \right\}
\nonumber \\ &=& \max_{(v_0^*, \hat{u}) \in B_r(\tilde{v}_0^*,u_0)} J_3^*(v_0^*,\hat{u}) \nonumber \\ &=&
J_3^*(\tilde{v}_0^*, u_0).
\end{eqnarray}
The proof is complete.

\end{proof}

\section{Conclusion} In the present work, firstly we have developed a duality principle  applied to a Ginzburg-Landau type system.
As an important feature of such result,  for the Hessian determinant (and Frech\'{e}t second derivatives) about the critical point in question, we have $$\det [\delta_{v_0^*,z^*}^2 \tilde{J}(v_0^*,z^*)]> \mathcal{O}\left(\frac{1}{\varepsilon^{3/2}}\right),$$ where $\varepsilon>0$ is a small parameter.

In a second step, we present another duality principle and related primal dual formulation, which seems to be very interesting from a computational point of view.

Finally, we emphasize the results obtained are applicable to a large class of problems, including models of plates, shells and other models in elasticity.


\begin{thebibliography}{}
\bibitem{1}
R.A. Adams and J.F. Fournier, Sobolev Spaces, 2nd edn.
 (Elsevier, New York, 2003).

 \bibitem{100}
J.F. Annet, Superconductivity, Superfluids and Condensates, 2nd edn.
 ( Oxford Master Series in
Condensed Matter Physics, Oxford University Press, Reprint, 2010)

 \bibitem{12}
F. Botelho, Functional Analysis and Applied Optimization in Banach Spaces,
 (Springer Switzerland, 2014).


\bibitem{101}
L.D. Landau and E.M. Lifschits, Course of Theoretical Physics, Vol. 5- Statistical Physics, part 1.
(Butterworth-Heinemann, Elsevier, reprint 2008).
\bibitem{29} R.T. Rockafellar,  Convex Analysis, Princeton Univ. Press,
(1970).



\bibitem{12a}
 J.F. Toland,  A duality principle for non-convex
optimisation and the calculus of variations, Arch. Rath. Mech.
Anal., {\bf 71}, No. 1 (1979), 41-61.


\end{thebibliography}
\end{document}